\documentclass[11pt,reqno]{amsart}
\usepackage{amsmath,amsthm}
\usepackage{amssymb,latexsym}
\usepackage{graphicx,subfig}
\usepackage{epsfig}
\usepackage{color}

\usepackage{lineno}
\newcommand*\patchAmsMathEnvironmentForLineno[1]{%
  \expandafter\let\csname old#1\expandafter\endcsname\csname #1\endcsname
  \expandafter\let\csname oldend#1\expandafter\endcsname\csname end#1\endcsname
  \renewenvironment{#1}%
     {\linenomath\csname old#1\endcsname}%
     {\csname oldend#1\endcsname\endlinenomath}}% 
\newcommand*\patchBothAmsMathEnvironmentsForLineno[1]{%
  \patchAmsMathEnvironmentForLineno{#1}%
  \patchAmsMathEnvironmentForLineno{#1*}}%
\AtBeginDocument{%
\patchBothAmsMathEnvironmentsForLineno{equation}%
\patchBothAmsMathEnvironmentsForLineno{align}%
\patchBothAmsMathEnvironmentsForLineno{flalign}%
\patchBothAmsMathEnvironmentsForLineno{alignat}%
\patchBothAmsMathEnvironmentsForLineno{gather}%
\patchBothAmsMathEnvironmentsForLineno{multline}%
}

\newtheorem{thm}{Theorem}[section]

\newtheorem{lma}[thm]{Lemma}
\newtheorem{cor}[thm]{Corollary}
\newtheorem{prp}[thm]{Proposition}

\newtheorem{clm}[thm]{Claim}
\newtheorem{rmk}[thm]{Remark}

\newtheorem{conj}[thm]{Conjecture}

\newtheorem{qns}[thm]{Question}

\DeclareMathOperator{\ex}{ex}

\title{Minimum codegree threshold for $(K_4^3-{e})$-factors}

\author{Allan Lo}
\address{School of Mathematics\\ University of Birmingham\\Birmingham\\ B15 2TT\\ UK}
\email{s.a.lo@bham.ac.uk}
\thanks{The first author was supported by the ERC, grant no.~258345.}

\author{Klas Markstr{\"o}m}
\address{Department of Mathematics and Mathematical Statistics\\ Ume\r{a} University\\ S-901 87 Ume\r{a}\\ Sweden}
\email{klas.markstrom@math.umu.se}

\date{\today}

\begin{document}
%\linenumbers

\begin{abstract}
Given hypergraphs $H$ and $F$, an $F$-factor in~$H$ is a spanning subgraph consisting of vertex-disjoint copies of~$F$.
Let $K_4^3-e$ denote the 3-uniform hypergraph on 4 vertices with 3 edges.
We show that for any $\gamma>0$ there exists an integer $n_0$ such that every 3-uniform hypergraph $H$ of order $n > n_0$ with minimum codegree at least $(1/2+\gamma)n$ and $4|n$ contains a $(K_4^3-e)$-factor.
Moreover, this bound is asymptotically the best possible and we further give a conjecture on the exact value of the threshold for the existence of a $(K_4^3-e)$-factor.
Thereby, all minimum codegree thresholds for the existence of $F$-factors are known asymptotically for 3-uniform hypergraphs $F$ on 4 vertices.
\end{abstract}

\maketitle

\section{Introduction} \label{sec:introduction}

Given hypergraphs $H$ and $F$, an \emph{$F$-factor} (or a \emph{perfect $F$-tiling} or a \emph{perfect $F$-matching}) in~$H$ is a spanning subgraph consisting of vertex-disjoint copies of~$F$.
Clearly, if $H$ contains an $F$-factor then $|V(F)|$ divides~$|V(H)|$.
A \emph{$k$-uniform hypergraph}, or \emph{$k$-graph} for short, is a pair $H = (V(H),E(H))$, where $V(H)$ is a finite set of vertices and $E(H)$ is a set of $k$-element subsets of $V(H)$.
If $H$ is known from the context, then we will often write $V$ instead of $V(H)$.
We often write $l$-sets for $l$-element subsets.
For a $k$-graph $H$ and an $l$-set $T \in \binom{V}l$, let $\deg(T)$ be the number of $(k-l)$-sets $S \in \binom{V}{k-l}$ such that $S \cup T$ is an edge in~$H$, and let $\delta_l(H)$ be the \emph{minimum $l$-degree} of $H$, that is, $\delta_l(H) = \min \{ \deg(T) : T \in \binom{V}l \}$.
Define $t_l^k(n,F)$ to be the smallest integer $d$ such that every $k$-graph $H$ of order $n$ with $\delta_l(H) \ge d$ contains an $F$-factor.
If $n$ is not divisible by $|V(F)|$, then $t_l^k(n,F) = \binom{n-l}{k-l}$.
Hence, we always assume that $|V(F)|$ divides $n$.

For graphs (that is, 2-graphs), a classical theorem of Hajnal and Szemer\'{e}di~\cite{MR0297607} states that $t_1^2(n,K_t) = (t-1)n/t$. 
Furthermore, $t_1^2(n,F)$ is known up to an additive constant for every 2-graph~$F$, see~\cite{MR2506388}.
For graphs $F$, there is a large body of research on $t^2_1(n, F)$.
For two surveys see~\cite{MR2588541,yuster2007combinatorial}. 

In the case of hypergraphs ($k\ge3$), only a few values of $t^k_l(n,F)$ are known.
Note that when $F$ is a single edge $K_k^k$, a $K_k^k$-factor is equivalent to a perfect matching.
R\"{o}dl, Ruci\'{n}ski and Szemer\'{e}di~\cite{MR2500161} proved that
\begin{align}
t^k_{k-1}(n,K_k^k) & =  \frac{n}2 -k +\varepsilon_n, \textrm{ where $\varepsilon_n \in \{ 3/2, 2,5/2,3\}$.} \nonumber
\end{align}
For $k > l \ge 1$, K\"{u}hn and Osthus~\cite{MR2588541} and independently H\'{a}n, Person and Schacht~\cite{MR2496914} conjectured that 
\begin{align*}
t_l^k(n,K_k^k) = \left( \max \left\{ \frac12, 1 - \left(1-\frac1k\right)^{k-l} \right\} +o(1) \right) \binom{n}k.
\end{align*}
This conjecture has been verified for various cases of $k$ and $l$.
We recommend \cite{rodldirac} for a survey on~$t_l^k(n, K_k^k)$.

Here we focus on the case when $k=3$, $l=2$ and $|V(F)|=4$.
Let $K_4^3$ be the complete 3-graph on 4 vertices. 
In~\cite{lo2011f}, the authors showed that $t^3_2(n,K_4^3) = (3/4+o(1))n$, and independently Keevash and Mycroft~\cite{keevash2011geometric} determined the exact value of $t^3_2(n,K_4^3)$ for sufficiently large $n$.
For $1 \le i \le 3$, let $K_4^3- i e$ be the unique 3-graph on 4 vertices with $(4-i)$ edges.
K\"{u}hn and Osthus~\cite{MR2274077} showed that $t^3_2(n,K_4^3-2e) = (1/4+o(1))n$, and the exact value was determined by Czygrinow, DeBiasio and Nagle~\cite{czygrinow2011tiling} for large~$n$.
Note that $K_4^3-3e$ is simply an edge plus one isolated vertex and so a $(K_4^3-3e)$-factor corresponds to a matching of size at least $n/4$.
By Fact~2.1 in~\cite{MR2500161}, it is easy to deduce that $t^3_2(n,K_4^3-3e) = n/4$.
In this paper, we investigate $t^3_2(n,K_4^3-e)$, the only remaining case for 3-graphs on 4 vertices.
It is easy to show that $t^3_2(4 ,K_4^3 - e) = 1$.
Also, we know that $t^3_2(8 ,K_4^3 - e) = 4$ by a computer search.
For $n \ge 12 $, we give the following lower bound on $t^3_2(n ,K_4^3 - e)$.
\begin{prp} \label{prp:lowerimproved}
For integers $n $ with $4|n$
\begin{align*}
t^3_2(n ,K_4^3 - e) \ge n/2 - 1.
\end{align*}
\end{prp} 

We show that the inequality above is indeed asymptotically sharp.

\begin{thm} \label{thm:t(n,K_4m)}
Given a constant $\gamma >0$, there exists an integer $n_0 = n_0(\gamma)$ such that for all $n \ge n_0$ with $4|n$, $t^3_2(n,K_4^3-e) \le \left(1/2+ \gamma \right)n$.
\end{thm}

We now present an outline of the proof of Theorem~\ref{thm:t(n,K_4m)}, which uses the absorption technique introduced by R\"odl, Ruci\'{n}ski and Szemer\'{e}di~\cite{MR2500161}.
First, we remove a set $U$ of vertex-disjoint copies of $K_4^3-e$ from $H$ satisfying the conditions of the absorption lemma, Lemma~\ref{lma:absorptionlemma}, and call the resulting graph $H'$ with $\delta_2(H') \ge |H'|/2$.
Next, we find vertex-disjoint copies of $K_4^3-e$ covering all but at most $12$ vertices of $H'$.
Let $W$ be the set of `leftover' vertices.
By the absorption property of $U$ there is a $(K_4^3-e)$-factor in $H[U \cup W]$.
Hence, we obtain a $(K_4^3-e)$-factor in $H$ as required.

We further conjecture that equality holds in Proposition~\ref{prp:lowerimproved}.

\begin{conj} \label{conjecture}
For integers $n > 8$ with $4|n$, $t^3_2(n ,K_4^3 - e) = n/2 - 1$.
\end{conj}

\section{Notations and preliminaries}

In the remainder of the paper, we will only consider 3-graphs unless stated otherwise. 
For simplicity, we write $K_4$ and $K_4^-$ for $K_4^3$ and $K_4^3 -e$ respectively.
We refer to the set $\{1, \dots, a\}$ as $[a]$ for $a \in \mathbb{N}$.

For a $3$-graph $H$ and a vertex set $U \subseteq V(H)$, $H[U]$ is the subgraph of $H$ induced by the vertices of~$U$.
We often write $v$ to mean the set $\{v\}$ when there is no risk for confusion.
For a $2$-set $T = \{u,v\}$, the \emph{neighbourhood $N(T)$ (or $N(u,v)$) of $T$} is the set of vertices $v$ such that $T \cup v$ is an edge in~$H$.
Hence, $\deg(T) = |N(T)|$ and similarly we write $\deg(u,v) = \deg(T)$ for $T = \{u,v\}$.
Let $V_1, \dots, V_l$ be a partition of~$V(H)$.
We say that an edge $v_1v_2v_3$ is \emph{of type $V_{i_1}V_{i_2}V_{i_3}$} if $v_j \in V_{i_j}$ for $ j \in [3]$ and denote the number of edges of type $V_{i_1}V_{i_2}V_{i_3}$ by $e(V_{i_1}V_{i_2}V_{i_3})$.
Similarly, we define \emph{types} for $K_4^-$, where the location of the vertex of degree~3 does not matter.
Given a 3-set~$T$, let $L(T)$ denote the set of vertices $v$ such that $H[T \cup v]$ contains a $K_4^-$.

\begin{prp} \label{prp:edgeextension}
Let $H$ be a $3$-graph of order~$n$.
Then for every edge $e$, $|L(e)| \ge (3 \delta_2(H) - n)/2$.
\end{prp}

\begin{proof}
Let $e = xyz$.
For $i=0,1,2,3$, let $n_i$ denote the number of vertices belonging to exactly $i$ neighbourhoods of $\{x,y\}$, $\{x,z\}$ and $\{y,z\}$.
For example, $n_3 = |N(x,y) \cap N(x,z) \cap N(y,z)|$.
Note that $\sum n_i = n$ and $\sum in_i \ge 3 \delta_2(H)$.
Thus, $2n_3 +n_2 \ge 3 \delta_2(H) - n$.
If a vertex $v$ is in at least two neighbourhoods of $\{x,y\}$, $\{x,z\}$ and $\{y,z\}$, then $H[\{x,y,z,v\}]$ contains a~$K_4^-$.
Thus, the proposition follows as $|L(e)| = n_2 +n_3$.
\end{proof}

The \emph{Tur\'an number} of $K_4^-$, $\ex(n,K_4^-)$, is the maximum number of edges in a $K_4^-$-free 3-graph of order $n$.
Currently, it is known that $ (2/7 +o(1)) \binom{n}3 \le \ex(n,K_4^-)  \le (0.2871+o(1) )\binom{n}3$, where the lower bound is due to Frankl and F{\"u}redi~\cite{MR753720} and the upper bound is due to Baber and Talbot~\cite{MR2769186}.
If $H$ is a 3-graph of order $n$ with $e(H) > \ex(n,K_4^-) + c n^3$, then we have the `supersaturation' phenomenon discovered by Erd{\H{o}}s and Simonovits~\cite{MR726456}.

\begin{thm}[Supersaturation] \label{thm:supersaturation}
For every constant $c>0$, there exists a constant $c'>0$ such that every 3-graph $H$ of order~$n$ with $e(H) > \ex(n,K_4^-) + c n^3$ contains at least $c' n^4$ copies of $K_4^-$.
For every constant $c>0$, there exists a constant $c'>0$ such that every 3-graph $H$ of order~$n$ with $e(H) > \ex(n,K_4^-) + c n^3$ contains at least $c' n^4$ copies of $K_4^-$.
\end{thm}

\begin{cor} \label{cor:supersaturation}
There exists a constant $c'>0$ such that every 3-graph $H$ of order~$n$ with $e(H) > 0.3 \binom{n}3$ contains at least $c' n^4$ copies of $K_4^-$.
\end{cor}

Given an integer $i \ge 1$ and vertices $x,y \in V(H)$, we say that the vertex set $S \subseteq V(H)$ is an \emph{$(x,y)$-connector of length $i$} if $S \cap \{x,y\} = \emptyset$, $|S| = 4i -1$ and both $H[S \cup x]$ and $H[S \cup y]$ contain $K_4^-$-factors.
Given an integer $i\ge 1$ and a constant $\eta>0$, two vertices $x$ and $y$ are $(i, \eta)$-close to each other if there exist at least $\eta n^{4i-1}$ $(x,y)$-connectors of length~$i$.
We denote by $\widetilde{N}_{i,\eta} (x)$ the set of vertices $y$ that are \emph{$(i, \eta)$-close} to $x$.
A subset $U \subseteq V$ is said to be \emph{$(i , \eta)$-closed in $H$} if every two vertices in $U$ are $(i, \eta)$-close to each other.
Moreover, $H$ is said to be \emph{$(i , \eta)$-closed} if $V(H)$ is $(i , \eta)$-closed in~$H$.
If $\eta$ is known from context, we simply write $i$-closed and $\widetilde{N}_{i} (x)$ for $(i, \eta)$-closed and $\widetilde{N}_{i,\eta} (x)$ respectively.
For $X,Y \subseteq V$, a triple $(x,y,S)$ is an \emph{$(X,Y)$-bridge of length $i$} if $x \in X$, $y \in Y$ and $S$ is an $(x,y)$-connector of length $i$.
If $u \in X \cap Y$, then we say $(u,u, \emptyset)$ is an $(X,Y)$-bridge of length~$0$.

Next we study some basic properties of $(i,\eta)$-closeness.

\begin{prp} \label{prp:i-closedadditive}
Let $i \ge 1$ be an integer and let $\eta, \varepsilon >0$ be constants.
Let $n$ be a sufficiently large integer and let $H$ be a $3$-graph of order $n$.
Suppose that $| \widetilde{N}_{i, \eta}(x)| \ge \varepsilon n$ for a vertex $x \in V$.
Then, $ \widetilde{N}_{i, \eta}(x) \subseteq  \widetilde{N}_{i+1, \eta'}(x)$ for some constant $\eta' >0$.
\end{prp}

\begin{proof}
Let $y \in \widetilde{N}_{i}(x)$ and $m = 4i-1$.
To prove the proposition, it is enough to show that $y$ is $(i+1,\eta')$-close to $x$ for some $\eta' >0$.
There are at least $\eta n^m$ $(x,y)$-connectors $S$ of length~$i$.
Fix an $(x,y)$-connector $S$ of length~$i$.
Let $z \in \widetilde{N}_{i}(x) \setminus ( S \cup \{x,y\} )$.
There are at least $\eta n^m$ $(x,z)$-connectors $S'$ of length~$i$.
Moreover, the number of $S'$ containing a vertex in $S \cup y$ is at most $(m+1) n^{m-1} < \eta n^m/2$.
Hence, there are at least $\eta n^m/2$ $(x,z)$-connectors $S'$ with $S' \cap (S \cup y) = \emptyset$.
Since $H[S' \cup z]$ contains a $K_4^-$-factor, there is a 3-set $T \subseteq S'$ such that $z \in L(T)$.
By an averaging argument, the number of $K_4^-$ vertex-disjoint from $S \cup \{x, y\}$ is at least 
\begin{align*}
\frac{\eta n^m/2}{n^{m-3}} \times \frac{\varepsilon n - m-2}4 > \eta \varepsilon n^4 /16.
\end{align*}
Recall that $S$ is an $(x,y)$-connector of length $i$, so if $U$ spans a $K_4^-$ in $H$ and $U \cap (S \cup \{x,y\}) = \emptyset $, then $S \cup U$ is an $(x,y)$-connector of length $i+1$.
Note also that there are 
\begin{align*}
	\frac{\eta n^m/2 \times \eta \varepsilon n^4 /16}{\binom{m+4}4} > \eta' n^{m+4}
\end{align*}
such choices $S \cup U$ for some constant $\eta' >0$.
Hence, $y$ is $(i+1,\eta')$-close to~$x$.
\end{proof}

\begin{lma} \label{lma:bridge}
Let $i_X,i_Y >0$ and $i \ge 0$ be integers and let $\eta_X, \eta_Y, \eta, \varepsilon >0$ be constants.
Let $n$ be a sufficiently large integer and let $H$ be a $3$-graph of order $n$.
Suppose that $x$ and $y$ are distinct vertices in $V(H)$.
Suppose there are at least $\varepsilon n^{4i+1}$ copies of $(X,Y)$-bridges of length $i$, where $X = \widetilde{N}_{i_X,\eta_X}(x)$ and $Y = \widetilde{N}_{i_Y,\eta_Y}(y)$.
Then, $x$ and $y$ are $(i_X+i_Y+i,\eta_0)$-close to each other for some $\eta_0 >0$.
In particular, if $|X \cap Y| \ge \varepsilon n$, then 
$x$ and $y$ are $(i_X+i_Y,\eta)$-close to each other for some $\eta >0$.

Furthermore, if $X$ and $Y$ are $(i_X,\eta_X)$-closed and $(i_Y,\eta_Y)$-closed in $H$ and $|X|,|Y| \ge \varepsilon n$, then $X \cup Y$ is $(i_X+i_Y+i,\eta)$-closed in~$H$.
\end{lma}

\begin{proof}
Let $i_0 = i_X+i_Y+i$ and let $\eta_0 >0$ be a sufficiently small constant.
Let $m_0 = 4i_0-1$, $m = 4i-1$, $m_X = 4i_X-1$ and $m_Y = 4i_Y-1$.
There are at most $(m+2)n^{m+1} < \varepsilon n^{m+2}$ copies of $(X,Y)$-bridges $(x',y',S)$ of length $i$ with $\{x,y\} \cap (S \cup \{x',y'\}) \ne \emptyset$.
Hence, the number of $(X,Y)$-bridges $(x',y',S)$ with $x' \in X \setminus (S \cup \{ x,y\})$ and $y' \in Y \setminus (S \cup \{ x,y\})$ is at least $\varepsilon n^{m+2}/2$.
Fix one such $(X,Y)$-bridge $(x',y',S)$.
Since $x' \in X \setminus x$, the number of $(x,x')$-connectors $S_X$ of length $i_X$ such that $S_{X} \cap (S \cup \{ x,x',y,y'\}) = \emptyset$ is at least 
\begin{align*}
\eta_X n^{m_X} - (m+4)   n^{m_X-1} \ge \eta_X n^{m_X}/2
\end{align*}
and fix one such $S_{X}$.
Similarly, the number of $(y,y')$-connectors $S_Y$ of length $i_Y$ such that $S_{Y} \cap (S \cup S_X \cup \{ x,x',y,y'\}) = \emptyset$ is at least 
\begin{align*}
\eta_Y n^{m_Y} - (m_X+m+4)n^{m_Y-1} \ge \eta_{Y} n^{m_Y}/2
\end{align*}
and fix one such $S_{Y}$.
Set $S_0 = S_{X} \cup S_{Y} \cup S \cup \{x',y'\}$.
Note that $S_0$ is an $(x,y)$-connector of length $i_0$.
Moreover, there are at least 
\begin{align}
	\frac{1}{\binom{m_0}{m,1,1,m_X,m_Y}} \times \frac{\varepsilon n^{m+2}}{2} \times \frac{\eta_{X} n^{m_X}}2  \times \frac{\eta_{Y} n^{m_Y}}2 \ge \eta n^{m_0} \nonumber
\end{align}
distinct $S_0$, so $x$ and $y$ are $(i_0,\eta_0)$-close to each other.
The second assertion holds as $(z,z,\emptyset)$ is an $(X,Y)$-bridge of length 0 for $z \in X \cap Y$.
Finally, the last assertion holds by Proposition~\ref{prp:i-closedadditive}.
\end{proof}

We now state the absorption lemma for $K_4^-$-factors, which is a special case of Lemma~1.1 in~\cite{lo2011f}.
We present its proof for completeness.

\begin{lma}[Absorption lemma] \label{lma:absorptionlemma}
Let $i \ge 1$ be an integer and let $\eta >0$ be a constant.
Then, there is an integer $n_0$ satisfying the following: Suppose that $H$ is a 3-graph of order $n \ge n_0$ and $H$ is $(i,\eta)$-closed.
Then there exists a vertex subset $U \subseteq V(H)$ of size $|U| \le \eta^{4} n/ (3 \times 2^8 i )$ such that $H[U \cup W]$ contains a $K_4^-$-factor for every vertex set $W \subseteq V \setminus U$ of size $|W|  \le \eta^{8} n/(2^{12} 3^2 i^2 ) $ with $|W|+|U| =0 \pmod{4}$.
\end{lma}

\begin{proof}
Let $H$ be a $3$-graph of order $n \ge n_0$ such that $H$ is $(i,\eta)$-closed.
Throughout the proof we may assume that $n_0$ is chosen to be sufficiently large.
Set $m_1 = 4i-1$ and $m= 3m_1+3 = 12i$.
Furthermore, call an $m$-set $A \in \binom{V}m$ an \emph{absorbing} $m$-set for a $4$-set $T\in \binom{V}4$ if $A \cap T= \emptyset$ and both $H[A]$ and $H[A \cup T]$ contain $K_4^-$-factors.
Denote by $\mathcal{L}(T)$ the set of all absorbing $m$-sets for $T$.
Next, we show that for every $4$-set $T$, there are many absorbing $m$-sets for~$T$.

\begin{clm} \label{clm:numberofabsorbingm-set}
For every $4$-set $T \in \binom{V}4$, $|\mathcal{L}(T)| \ge (\eta/2)^{4}\binom{n}{m}$.
\end{clm}

\begin{proof}
Let $T= \{v_1, v_2, v_3, v_4\}$ be a fixed 4-set.
Since $v_1$ and $u$ are $(i,\eta)$-connected for $u \notin T$, the number of $m_1$-sets $S$ such that $H[S \cup v_1]$ contains a $K_4^-$-factor is at least $\eta n^{m_1}$.
Hence, by an averaging argument there are at least $\eta n^{3}$ copies of $K_4^-$ containing~$v_1$.
Since $n_0$ is large, there are at most $3n^2 \le \eta n^{3}/2$ copies of $K_4^-$ containing $v_1$ and $v_j$ for some $2 \le j \le 4$.
Thus, there are at least $\eta n^{3}/2$ copies of $K_4^-$ containing $v_1$ but none of $v_2$, $v_3$, $v_{4}$.
We fix one such copy of $K_4^-$ with $V(K_4^-) = \{ v_1, u_2, u_3, u_4\}$.
Set $U_1 = \{ u_2, u_3, u_4\}$ and $W_0 = T$.

For each $2\le j \le 4$ and each pair $u_j, v_j$ suppose we have succeed in choosing an $m_1$-set $U_j$ such that $U_j$ is disjoint from $W_{j-1} = U_{j-1} \cup W_{j-2}$ and both $H[U_j \cup  u_j ]$ and $H[U_j \cup  v_j]$ contain $K_4^-$-factors.
Then for a fixed $2 \le j \le 4$ we call such a choice $U_j$ \emph{good}, motivated by $A = \bigcup_{1 \le j\le 4} U_j$ being an absorbing $m$-set for~$T$.

In each step $2\le j \le 4$, recall that $u_j$ is $(i,\eta)$-closed to $v_j$, so the number of $m_1$-sets $S$ such that $H[S \cup u_j]$ and $H[S \cup v_j]$ contain $K_4^-$-factors is at least $\eta n^{m_1}$.
Note that there are $7+(j-2)m_1$ vertices in $W_{j-1}$.
Thus, the number of such $m_1$-sets $S$ intersecting $W_{j-1}$ is at most 
\begin{align*}
(7+(j-2)m_1)n^{m_1-1} \le (7+2 m_1)n^{m_1-1} < \eta n^{m_1}/2.
\end{align*}
For each $2\le j \le 4$ there are at least $\eta n^{m_1}/2$ choices for $U_j$ and in total we obtain $(\eta/2)^{4}n^{m}$ absorbing $m$-sets for $T$ with multiplicity at most $m!$, so the claim holds.
\end{proof}

Now, choose a family $\mathcal{F}$ of $m$-sets by selecting each of the $\binom{n}{m}$ possible $m$-sets independently at random with probability $p = \eta^{4} n / (2^{7} m^2 \binom{n}{m})$.
Then, by Chernoff's bound (see e.g.~\cite{MR1885388}) with probability $1-o(1)$ as $n \rightarrow \infty$, the family $\mathcal{F}$ satisfies the following properties:
\begin{align}
|\mathcal{F}| \le & \eta^{4} n / (2^6 m^2) \label{eqn:|F|}
\intertext{and}
|\mathcal{L}(T) \cap \mathcal{F}| \ge & \eta^8 n/ (2^{12}m^2) \label{eqn:L(T)}
\end{align} 
for all $4$-sets $T$.
Furthermore, we can bound the expected number of intersecting $m$-sets by
\begin{align}
	\binom{n}{m} \times m \times \binom{n}{m-1} \times p^2 \le \frac{3\eta^8 n }{2^{14}m^2}.
\nonumber
\end{align}
Thus, using Markov's inequality, we derive that with probability at least~$1/2$
\begin{align}
	\textrm{$\mathcal{F}$ contains at most $\frac{\eta^8 n }{2^{13}m^2}$ intersecting pairs.} \label{eqn:F}
\end{align}
Hence, with positive probability the family $\mathcal{F}$ has all properties stated in \eqref{eqn:|F|}, \eqref{eqn:L(T)} and \eqref{eqn:F}.
By deleting all the intersecting $m$-sets and non-absorbing $m$-sets in such a family $\mathcal{F}$, we get a subfamily $\mathcal{F}'$ consisting of pairwise vertex-disjoint $m$-sets, which satisfies
\begin{align}
|\mathcal{L}(T) \cap \mathcal{F}'| \ge & \frac{\eta^8 n }{2^{12}m^2}- \frac{3\eta^8 n }{2^{14}m^2} = \frac{\eta^8 n }{2^{14}m^2} \nonumber
\end{align}
for all $4$-sets $T$.
Set $U = V(\mathcal{F}')$ and so $|U| \le \eta^4 n /(2^6 m)$ by~\eqref{eqn:|F|}.
Since $\mathcal{F}'$ consists only of absorbing $m$-sets, $H[U]$ has a $K_4^-$-factor.
So $|U| = 0 \pmod4$.
For any set $W \subseteq V \backslash U$ of size $|W| \le \eta^8 n/ (2^{12}m^2)$ and $ |W| \in 4 \mathbb{Z}$, $W$ can be partition into at most $\eta^8 n/ (2^{14}m^2)$ $4$-sets.
Each $4$-set can be successively absorbed using a different absorbing $m$-set, so $H[U \cup W]$ contains a $K_4^-$-factor. 
\end{proof}

\section{A lower bound on $t_2^3(n,K_4^3-e)$} \label{sec:lowerbound}

In this section, we are going to bound $t_2^3(n,K_4^-)$ from below, thereby proving Proposition~\ref{prp:lowerimproved}.

\begin{proof}[Proof of Proposition~\ref{prp:lowerimproved}]
For integers $a,b>0$,  let $A$ and $B$ be two disjoint vertex sets with $|A| = a$ and $|B| = b$.
We define a 3-graph $H_{a,b}$ on the vertex set $A \cup B$ such that every edge contains odd number of vertices in $B$.
Hence, every edge in $H_{a,b}$ is of type $AAB$ or $BBB$.
Note that $\delta_2(H_{a,b}) = \min\{ b,a-1,b-2 \}$ by considering $\deg(v,v')$, $\deg(v,w)$, $\deg(w,w')$ for distinct $v,v' \in A$ and distinct $w , w' \in B$.
Moreover, every $K_4^-$ in $H_{a,b}$ is of type $AAAB$ or $BBBB$ and so every $K_4^-$ in $H_{a,b}$ contains exactly 0 or 3 vertices of~$A$.
Thus, $H_{a,b}$ does not contain a $K_4^-$-factor if $a \ne 0 \pmod{3}$.

Recall that $n = 0 \pmod{4}$.
If $n \ne 0 \pmod{3}$, then $t_2^3(n, K_4^-) >  n/2 -2$ by considering $H_{n/2, n/2}$.
If $n = 0 \pmod{3}$, then $t_2^3(n, K_4^-) >  n/2 -2$ by considering $H_{n/2-1, n/2+1}$.
\end{proof}

\begin{rmk} \label{rmk:lowerimproved}
Actually, to show that $t_2^3(n,K_4^-) \ge n/2-1$ for $n = 1 \pmod{3}$, we could consider $H_{n/2-1, n/2 + 1}$ instead of $H_{n/2, n/2}$.
This can be done since $n/2 = 2 \pmod{3}$ and so $n/2-1 \ne 0 \pmod{3}$.
In fact, for $n = 1 \pmod3$, we can define a family of $3$-graphs $H$ with $\delta_2(H) = n/2-2$ with no $K_4^-$-factors as follows.
Let $A = \{v_1, \dots, v_{n/2-1} \}$ and $B = \{ w_1, \dots, w_{n/2}\}$ be two disjoint vertex sets.
Let $z$ be a vertex disjoint from $A$ and $B$.
For a given integer $1 \le l \le n/2$, define $H_l$ to be the $3$-graph on $A \cup B \cup  z $ with edge set $E(H_l) = E_1 \cup E_2 \cup E_3$ such that
\begin{align*}
	E_1 & = \left\{ T \in \binom{A \cup B}3 : |T \cap B|  = 1 \pmod{2} \right\},\\
	E_2 & = \{ z v_i v_j, z w_i w_j : i < \min\{ j,  l\}  \},\\
	E_3 & = \{ z v_i w_j : l\le \min\{ i , j\}  \}.
\end{align*}
(Notice that $H_1 = H_{n/2,n/2}$ and $H_{n/2-1} = H_{n/2-1,n/2+1}$.)
Note that $N(z,v_i) = A \setminus v_i$ for $i <l$ and $N(z, v_i) = \{v_1, v_2, \dots, v_{l-1}, w_l, w_{l+1}, \dots, w_{n/2} \}$ for $i \ge l$.
Thus, $\deg(z,v) \ge |A| -1 = n/2 -2$ for $v \in A$, and by a similar argument $\deg(z,w) \ge n/2 -2$ for $w \in B$.
Note that $H_l[A \cup B]$ is isomorphic to $H_{n/2-1, n/2}$.
Hence, $\delta(H_l) = n/2-2$.

Next, we are going to show that $H_l$ does not contain a $K_4^-$-factor.
Suppose the contrary, $H_l$ contains a $K_4^-$-factor.
Note that every $K_4^-$ in $H_l[A \cup B]$ is of type $AAAB$ or $BBBB$.
Since $|A| = n/2-1 = 1 \pmod{3}$ and $H_l$ contains a $K_4^-$-factor, there exists a $K_4^-$ with vertex set $\{z,v_i,w_j,w_k\}$ for some $i,j,k \in [n/2]$ with $j < k$.
Note that $v_iw_jw_k$ is not an edge in $H_l$, so $zv_iw_j$, $zv_iw_k$, $zw_jw_k$ are edges in $H_l$.
By the definition of $E_2$, we deduce that $j < l$ as $zw_jw_k \in E(H_l)$.
This is a contradiction as $zv_iw_j \in E(H_l)$.
Therefore $H_l$ does not contain a $K_4^-$-factor.
\end{rmk}

\section{An upper bound on $t_2^3(n,K_4^3-e)$} \label{sec:upperbound}

In the next theorem, we study the relationship between $\delta_2(H)$ and the number of the vertex-disjoint copies of $K_4^-$ in~$H$.
Note that $|V(H)|$ is not assumed to be divisible by 4 in the hypothesis.

\begin{thm}	\label{thm:approximate}
Let $l$ and $n$ be  integers with $0 \le l \le (n-13)/4$.
Let $H$ be a $3$-graph of order $n$ with $\delta_2(H) > (n +2l-2)/3$.
Then, there exist at least $l$ vertex-disjoint copies of $K_4^-$ in~$H$.
\end{thm}

\begin{proof}
Let $\mathcal{T}$ be a set of vertex-disjoint copies of $K_4^-$ and edges in $H$.
Let $\mathcal{T}_1$ and $\mathcal{T}_2$ be the set of $K_4^-$ and edges of~$\mathcal{T}$ respectively.
If $|\mathcal{T}_1| \ge l$, then we are done.
Hence, we may assume that $|\mathcal{T}_1| < l$ for all $\mathcal{T}$.
We define the weighting $w (\mathcal{T})$ of $\mathcal{T}$ to be $w(\mathcal{T}) = 5|\mathcal{T}_1| + 2 |\mathcal{T}_2|$.
We assume that $\mathcal{T}$ is chosen such that $w(\mathcal{T})$ is maximum.

First, we are going to show that $|\mathcal{T}_2| < 4$.
Suppose the contrary, so there are 4 disjoint edges $e_1, e_2, e_3 , e_4 \in \mathcal{T}_2$.
Note that if $v  \in L(e_i)$ for some $1 \le i \le 4$, then $v \in V(\mathcal{T}_1)$.
Otherwise, $\mathcal{T}'  =  (\mathcal{T} \setminus \{e_i, e_0\}) \cup \{ V(e_i) \cup  v\}$ contradicts the maximality of $w(\mathcal{T})$, where $e_0$ is the edge in $\mathcal{T}_2$ that contains $v$ if it exists.
By Proposition~\ref{prp:edgeextension}, $|L(e_i)| \ge  (3 \delta_2(H) -n)/2 > l-1$ for $i \in [4]$.
Thus, there exists $S = \{v_1,v_2,v_3,v_4\} \in \mathcal{T}_1$ such that $\sum_{i \in [4]} |L(e_i) \cap S| \ge 5$.
Without loss of generality, we may assume by the K\"{o}nig-Egerv\'{a}ry Theorem (see~\cite{MR2368647} Theorem~8.32) that $v_1 \in L(e_1)$ and $v_2 \in L(e_2)$.
Set $\mathcal{T}'  =  (\mathcal{T} \setminus \{S, e_1,e_1\}) \cup \{ V(e_1) \cup  v_1, V(e_2) \cup  v_2\}$.
Note that 
\begin{align*} 
w(\mathcal{T}') = w(\mathcal{T}) -(5+2+2)+(5+5) = w(\mathcal{T})+1,
\end{align*}
a contradiction.
Thus, we have $|\mathcal{T}_2| <4$.

Note that 
\begin{align*}
|V \setminus V(\mathcal{T})| \ge n - 4|\mathcal{T}_1| - 3 |\mathcal{T}_2| \ge n - 4(l-1) - 9 = n - 4l -5 \ge 8.
\end{align*}
Let $x_1$, \dots, $x_4$, $y_1$, \dots, $y_4$ be distinct vertices in $V \setminus V(\mathcal{T})$.
Since $w(\mathcal{T})$ is maximum, $N(x_i,y_i) \subseteq V(\mathcal{T})$.
If $\sum_{i \in [4]} |N(x_i,y_i) \cap V(\mathcal{T}_2)| > 4|\mathcal{T}_2|$, there exists an edge $e \in \mathcal{T}_2$ such that $\sum_{i \in [4]} |N(x_i,y_i) \cap V(e)| \ge 5$.
By the K\"{o}nig-Egerv\'{a}ry Theorem, we may assume that $x_1y_1 v_1$ and $x_2y_2v_2$ are edges for distinct vertices $v_1, v_2 \in V(e)$.
Hence, $w(\mathcal{T}') = w(\mathcal{T}) +2$, where $\mathcal{T}'  =  \mathcal{T} \setminus e \cup \{x_1y_1 v_1 , x_2y_2v_2 \}$, a contradiction.
Therefore, $\sum_{i \in [4]} |N(x_iy_i) \cap V(\mathcal{T}_2)| \le 4|\mathcal{T}_2|$.
Recall that $|\mathcal{T}_2| \le 3$ and so
\begin{align*}
	\sum_{i \in [4]} |N(x_i,y_i) \cap V(\mathcal{T}_1)| \ge 4 \delta_2(H) - 12 > 8 |\mathcal{T}_1|.
\end{align*}
By an averaging argument, there exists $S = \{v_1,v_2,v_3,v_4\} \in \mathcal{T}_1$ such that $\sum |N(x_i,y_i) \cap S| \ge 9$.
Again by the K\"{o}nig-Egerv\'{a}ry Theorem, we may assume without loss of generality that $x_iy_iv_i$ is an edge for $i \in [3]$.
Set
\begin{align}
\mathcal{T}'  = ( \mathcal{T} \setminus S )\cup \{x_1y_1v_1, x_2y_2v_2, x_3y_3v_3\}. \nonumber
\end{align}
Note that $w(\mathcal{T}') - w(\mathcal{T}) \ge 3 \times 2 -5 = 1$, a contradiction.
This completes the proof of the theorem.
\end{proof}

Next, we are going to prove Theorem~\ref{thm:t(n,K_4m)}. 
We proceed by the absorption technique of R\"{o}dl, Ruci\'{n}ski and Szemer\'{e}di~\cite{MR2500161}.
We require the following lemma, which is proven in Section~\ref{sec:(i,eta_i)-closed}.
\begin{lma} \label{lma:(i,eta)-close}
Let $\gamma > 0$ and let $H$ be a $3$-graph of sufficiently large order $n$ with $\delta_2(H) \ge (1/2+\gamma) n$.
Then, $H$ is $(i,\eta)$-closed for some integer $i$ and constant $\eta >0$.
\end{lma}

\begin{proof}[Proof of Theorem~\ref{thm:t(n,K_4m)}]
Let $\gamma>0$ and let $H$ be a 3-graph $H$ of sufficiently large order $n$ with $4|n$ and $\delta_2(H) \ge (1/2+\gamma)n$.
In order to prove Theorem~\ref{thm:t(n,K_4m)}, it is enough to show that $H$ contains a $K_4^-$-factor. 
By Lemma~\ref{lma:(i,eta)-close}, $H$ is $(i,\eta)$-closed for some $i$ and $\eta >0$.
We may further take $\eta$ to be sufficiently small ($\eta^4/(3 \times 2^8 i ) < \gamma$ would do).
Let $U$ be the vertex set given by Lemma~\ref{lma:absorptionlemma} and so $|U| \le \eta^4 n /(3 \times 2^8 i)$.
Let $H' = H[V(H) \setminus U]$.
Note that 
\begin{align}
	\delta_{2}(H') \ge (1/2 + \gamma - \eta^4/ ( 3 \times 2^8 i ) )n \ge n'/2 \nonumber
\end{align}
where $n' = n - |U|$.
There exists a family $\mathcal{T}$ of vertex-disjoint copies of $K_4^-$ in $H'$ covering all but at most $16$ vertices by Theorem~\ref{thm:approximate}.
Let $W = V(H') \setminus V(\mathcal{T})$, so $|W| \le 16$.
By Lemma~\ref{lma:absorptionlemma}, there exists a $K_4^-$-factor $\mathcal{T}'$ in $H[U \cup W]$.
Thus, $\mathcal{T} \cup \mathcal{T}'$ is a $K_4^-$-factor in~$H$.
\end{proof}

\section{Proof of Lemma~\ref{lma:(i,eta)-close}.} \label{sec:(i,eta_i)-closed}

Let $\gamma > 0$ and let $H$ be a $3$-graph of sufficiently large order $n$ with $\delta_2(H) \ge (1/2+\gamma) n$.
Our aim is to show that $H$ is $(i,\eta)$-closed for some $i$ and $\eta >0$ proving Lemma~\ref{lma:(i,eta)-close}.
Its proof is divided into the following steps.
First we show that we can partition $V(H)$ into at most 3 vertex classes such that each class is $(\lceil 4/ \gamma \rceil +2, \eta)$-closed in~$H$ and has size at least $n/4$.
If there is only one vertex class, then we are done.
When there are two or three vertex classes, we show that $H$ is $(i',\eta')$-closed using Lemma~\ref{lma:2parts} and Lemma~\ref{lma:3parts} respectively for some integer $i'$ and constant $\eta'>0$.

Recall that $\widetilde{N}_{i,\eta}(v)$ is the set of vertices that are $(i,\eta)$-closed to~$v$.
First, we show that the size of $\widetilde{N}_{1,\gamma^2/12}(v)$ is at least $(1/4+\gamma)n$ for every $v \in V$.

\begin{prp} \label{prp:1connectionradius}
Let $\gamma > 0$ and let $H$ be a $3$-graph of order $n > 8/\gamma$ with $\delta_2(H) \ge (1/2+\gamma) n$.
Then, for $v \in V$ there are at least $(1/4+\gamma)n$ vertices $y$ such that $y$ is $(1,\gamma^2/12)$-close to $v$.
\end{prp}

\begin{proof}
Write $\delta = \delta_2(H)$ and $V' = V \setminus v$.
Let $\{x,y\} \in N(v)$, i.e. $vxy$ is an edge.
Note that there are at least $\delta (n-1)/2 \ge n^2/4$ such pairs.
For $z \in N(x,y) \cap N(v,x)$, $H[\{v,x,y,z\}]$ contains a $K_4^-$.
Since $|N(x,y) \cap N(v,x)| \ge 2 \gamma n$, there are at $\gamma n^3/6$ edges $e=xyz$ such that $v \in L(e)$.

Let $G$ be a bipartite 2-graph with the following properties.
The vertex classes of $G$ are $V'$ and $E'$, where $E'$ is a set of edges $e$ such that $v \in L(e)$.
For $y \in V'$ and $e \in E'$, $\{y,e\}$ is an edge in $G$ if and only if $y \in L(e)$.
Note that $|E'| \ge \gamma n^3/6$.
For $e \in E'$
\begin{align}
d^G(e) = |L(e) \setminus v|  \ge \left(1/4+ 3\gamma/2 \right) n -1 > \left(1/4+ 11\gamma/8 \right) n\nonumber
\end{align}
by Proposition~\ref{prp:edgeextension}.
We claim that there are more than $(1/4+\gamma) n$ vertices $y \in V'$ with $d^G(y) \ge \gamma |E'|/2$.
Indeed, it is true or else we have
\begin{align}
\left(1/4+ 11\gamma/8 \right) n |E'|
< e(G) \le \gamma |E'|/2 \times (3/4 - \gamma) n + |E'| (1/4+ \gamma)n, \nonumber
\end{align}
a contradiction.
Note that $y \in V'$ is $(1, d^G(y)/n^{3})$-close to $v$, so the proposition follows.
\end{proof}

We are going to partition $V$ into at most three classes such that each class is of size at least $(1/4+\gamma)n$ and $( \lceil 4/\gamma \rceil+2, \eta)$-closed in $H$ for some $\eta >0$.

\begin{lma} \label{lma:connectioncomponent}
Let $\gamma > 0$ and let $H$ be a $3$-graph of order $n$ with $\delta_2(H) \ge (1/2+\gamma) n$.
Then, there exist a constant $\eta>0$ and a vertex partition of $V$ into at most three classes such that each class $W$ is $(\lceil 4/ \gamma \rceil +2, \eta)$-closed in $H$ and $|W| \ge (1/4+3\gamma/4) n$.
\end{lma}

\begin{proof}
Throughout this proof, $\eta_1$, \dots, $\eta_{\lceil 4/ \gamma \rceil + 2}$ is assumed to be a decreasing sequence of strictly positive sufficiently small constants.
We write $i$-close to mean $(i ,\eta_i)$-close and recall that $\widetilde{N}_i(x)$ is the set of vertices $y$ that are $i$-close to $x$.
If $|\widetilde{N}_2(v)| \ge (1+\gamma)n/2$ for all $v \in V$, then $|\widetilde{N}_2(v) \cap \widetilde{N}_2(u)| \ge \gamma n$ for $u, v \in V$.
Thus, $H$ is 4-closed by Lemma~\ref{lma:bridge}.
Hence, we may assume that there exists a vertex~$v$ such that $|\widetilde{N}_2(v)| < (1+\gamma)n/2$.
Let $U$ be the set of vertices $u \in \widetilde{N}_1(v)$ such that 
\begin{align*}
|\widetilde{N}_1(u) \cap  \widetilde{N}_{2} (v)| \ge(1/4+ \gamma/3)n.
\end{align*}

\begin{clm}
The size of $U$ is at least $(1 +3\gamma)n/4$ and $U$ is 2-closed in~$H$.
\end{clm}

\begin{proof}[Proof of claim]
Note that if $|\widetilde{N}_{1}(w) \cap  \widetilde{N}_1 (v)| \ge \gamma^2n/6$ for $w \in V \setminus v$, then $w \in \widetilde{N}_{2} (v)$ by Lemma~\ref{lma:bridge}.
Thus, for each $w \notin \widetilde{N}_{2}(v)$, 
\begin{align}
|\widetilde{N}_1(v) \cap \widetilde{N}_1 (w)| < \gamma^2n/6. \nonumber
\end{align}
Therefore, by summing over all $w \notin  \widetilde{N}_{2}(v)$, we have
\begin{align}
 \sum_{u \in \widetilde{N}_1(v)} |\widetilde{N}_1(u) \setminus \widetilde{N}_{2}(v)| & = 
\sum_{w  \notin \widetilde{N}_{2}(v)} |\widetilde{N}_1(v) \cap \widetilde{N}_1 (w)| 
< \gamma^2 n^2 /6.
\label{eqn:sumw}
\end{align}
Since $|\widetilde{N}_1(u')| \ge (1/4+\gamma) n$ for $u' \in V$ by Proposition~\ref{prp:1connectionradius}, for $u' \in \widetilde{N}_1(v) \setminus U$
\begin{align*}
|\widetilde{N}_1(u') \setminus \widetilde{N}_{2}(v) | =  |\widetilde{N}_1(u')| - |\widetilde{N}_1(u') \cap \widetilde{N}_{2}(v)| > {2\gamma n}/3.
\end{align*}
Therefore, by summing over $u' \in \widetilde{N}_1(v) \setminus U$ and \eqref{eqn:sumw}, we have
\begin{align*}
	{2\gamma n}  |\widetilde{N}_1(v) \setminus U| /3 \le & \sum_{u' \in \widetilde{N}_1(v) \setminus U} |\widetilde{N}_1(u') \setminus \widetilde{N}_{2}(v)| \le \sum_{u \in \widetilde{N}_1(v)} |\widetilde{N}_1(u) \setminus \widetilde{N}_{2}(v)| < {\gamma^2 n^2}/{6}.
\end{align*}
Again recall Proposition~\ref{prp:1connectionradius} that $|\widetilde{N}_1(v)| \ge (1/4 + \gamma)n$, so $|U| \ge (1 +3\gamma)n/4$ as desired.
Furthermore, for $u, u' \in U$, we have 
\begin{align}
|\widetilde{N}_1(u) \cap  \widetilde{N}_1 (u')| \ge  |\widetilde{N}_1(u) \cap \widetilde{N}_{2}(v)| + |\widetilde{N}_1(u') \cap \widetilde{N}_{2}(v)|- |\widetilde{N}_{2}(v)|  
\ge \gamma n /6 \nonumber
\end{align}
as $|\widetilde{N}_{2}(v)| < (1+\gamma)n/2$.
Hence, $u$ and $u'$ are $2$-close to each other by Lemma~\ref{lma:bridge}.
\end{proof}

Set $U_0 = U$.
For an integer $i \ge 1$, we define $U_{i}$ to be the set of vertices $u' \notin  W_{i-1}$ such that $|\widetilde{N}_1(u') \cap W_{i-1}| \ge \gamma n/4 $, where $W_{j'}$ is the set $\bigcup_{j=0}^{j'} U_j$.
By Lemma~\ref{lma:bridge} and an induction on $i$, we deduce that $H[W_i]$ is $(i+2)$-closed in~$H$.
Let $i_0$ be the smallest integer such that $|U_{i_0}| < \gamma n/4$.
Since $U_0, U_1, \dots$ are disjoint sets, $1 \le i_0 \le \lceil 4 / \gamma\rceil$.
If $W_{i_0} = V(H)$, then $H$ is $(i_0+2)$-closed and so $H$ is $(\lceil 4 / \gamma\rceil+2)$-closed by Proposition~\ref{prp:i-closedadditive}.
Thus, we may assume that $V(H) \ne W_{i_0}$.
Note that $|W_{i_0}| \ge |U| \ge (1+ 3\gamma)n/4$.
For every $w \notin W_{i_0}$,  we have 
\begin{align*}
	|\widetilde{N}_{1}(w) \setminus W_{i_0}| & \ge |\widetilde{N}_{1}(w)| - 	\left|\widetilde{N}_{1}(w) \cap W_{i_0-1} \right| - |U_{i_0}| \\
& \ge  (1/4+\gamma)n - \gamma n/4 -  \gamma n/4
	=  (1/4 + \gamma/2 )n.
\end{align*}
Let $V' = V \setminus W_{i_0}$.
Note that $|V'| \le 3n/4$ and $|\widetilde{N}_1(u) \cap V'| \ge (1/4 + \gamma/2 )n$ for all $u \in V'$.
Thus, we are done by repeating the whole argument at most twice by replacing $V$ with $V'$.
\end{proof}

To prove Lemma~\ref{lma:(i,eta)-close}, it is sufficient to consider the case when there are two or three partition classes satisfying the conditions in Lemma~\ref{lma:connectioncomponent}.
Recall that an $(X,Y)$-bridge of length~$i$ is a triple $(x,y,S)$ such that $x \in  X$, $y \in Y$ and $S$ is an $(x,y)$-connector of length $i$.
To prove Lemma~\ref{lma:(i,eta)-close}, it is enough by Lemma~\ref{lma:bridge} to show that there are at least $\varepsilon n^{4i+1}$ $(X,Y)$-bridges of length~$i$ for some $\varepsilon >0$, where $X$ and $Y$ are the partition classes satisfying the conditions in Lemma~\ref{lma:connectioncomponent}.

We need the lemma below. 
Recall that $L(e)$ is the set of vertices $v$ such that $V(e) \cup v$ spans a $K_4^-$ in $H$ and $|L(e)| \ge (1/4+\gamma)n$ by Proposition~\ref{prp:edgeextension}.

\begin{lma} \label{lma:XYbridge}
Let $\gamma, c_1, c_2,c_3, c_4, \varepsilon_1, \varepsilon'_2, \varepsilon_2, \varepsilon_3,\varepsilon_3' \varepsilon_4 >0$ be constants such that 
\begin{align*}
\varepsilon_1 &<  \min\{ \varepsilon_2, \varepsilon_3 \}, & 	
c_1 +\varepsilon_2 & < c_2 < c_3 \varepsilon_3', \\
	\max\{ 2 \varepsilon_1 + \varepsilon_3', 4\varepsilon'_2 \}& < 3 \gamma, & 
	2c_1 & < c_3 <  \min\{ c_4\varepsilon_4/2 - \varepsilon_3 \}.
\end{align*}
Let $n$ be a sufficiently large integer and let $H$ be a $3$-graph of order $n$ with $\delta_2(H) \ge (1/2+\gamma) n$.
Suppose that $V(H)$ is partitioned into $X$ and $Y$ with $n/4 \le |X| \le n/2 \le |Y| $.
Furthermore, at least one of the following conditions holds:
\begin{itemize}
	\item[(i)] there are $c_1 n^{3}$ edges $e$ such that $|L(e) \cap X| \ge \varepsilon_1 n$ and $|L(e) \cap Y| \ge \varepsilon_1 n$,
	\item[(ii)] there are $c_2 n^{4}$ copies $T$ of $K_4$ such that $|T \cap X | = 2 = |T \cap Y|$,
	\item[(iii)] there are $c_3 n^{3}$ edges $xyy'$ of type $XYY$ such that $|L(xyy')\cap X | \ge \varepsilon_3 n$,
	\item[(iv)] there are $c_4 n^{3}$ edges $xx'y$ of type $XXY$ such that $|L(xx'y)\cap Y | \ge \varepsilon_4 n$.
\end{itemize}
Then, there exists $\varepsilon \ge 0$ such that the number of $(X,Y)$-bridges of length $1$ is at least $\varepsilon n^{5}$.
\end{lma}

\begin{proof}
Write $\delta = \delta_2(H)$.
We consider each condition one by one.

(i) There exist $c_1 n^{3}$ edges $e$ such that $|L(e) \cap X| \ge \varepsilon_1 n$ and $|L(e) \cap Y| \ge \varepsilon_1 n$.
For each such edge $e$, $(x,y,V(e))$ is an $(X,Y)$-bridge for $x \in L(e) \cap X$ and $y \in L(e) \cap Y$.
Therefore, there are at least $c_1 \varepsilon^2_1 n^{5}$ $(X,Y)$-bridges of length~$1$.

(ii) There exist $c_2 n^{4}$ copies $T$ of $K_4$ such that $|T \cap X | = 2 = |T \cap Y|$.
There are at least $(c_2-\varepsilon_2) n^3$ edges $e$ of type $XXY$ contained in at least $\varepsilon_2 n $ copies of these $K_4$.
Otherwise, the number of these $K_4$ is at most 
\begin{align}
(c_2-\varepsilon_2 ) n^3 \times n + (1-c_2+ \varepsilon_2 ) n^3 \times \varepsilon_2 n < c_2 n^4, \nonumber
\end{align}
a contradiction.
Note that for each such edge $e$, $|L(e) \cap Y| \ge \varepsilon_2 n$.
By~(i), we may assume that there are at least $(c_2-\varepsilon_2-c_1) n^3$ edges $e$ of type $XXY$ contained in at least $\varepsilon_2 n $ copies of these $K_4$ with $|L(e) \cap X| \le \varepsilon_1 n $.
Fix one such edge $xx'y$ and let $y' \in Y$ such that $H[\{x,x',y,y'\}]$ is a $K_4$.
Note that there are $(c_2-\varepsilon_2-c_1) \varepsilon_2 n^4/2$ choices for $x$, $x'$, $y$ and~$y'$.

\begin{clm} \label{clm:K4}
One of $L(xx'y) \cap X$, $L(xx'y') \cap X$, $L(xyy') \cap Y$, $L(x'yy') \cap Y$ is of size at least $\varepsilon_2' n$.
\end{clm}

\begin{proof}[Proof of claim]
Suppose that the claim is false.
Note that
\begin{align*}
 2\varepsilon_2' n & \ge 2 |L(xx'y) \cap X|\\
 & \ge |N(x,x') \cap X| + |N(x,y) \cap X| + |N(x'y) \cap X| - |X|\\
 & \ge |N(x,y) \cap X| + |N(x',y) \cap X| - |X|.
\end{align*}
Since $|N(x,y) \cap X| \ge \delta - |N(x,y) \cap Y|$ and $|N(x',y) \cap X| \ge \delta - |N(x',y) \cap Y|$, we have
\begin{align}
	     |N(x,y) \cap Y| +  |N(x',y) \cap Y| & \ge  2 \delta - |X|-2\varepsilon_2' n     . \label{eqn:L(xx'y)}
\end{align}
Similarly,
\begin{align}
|N(x,y') \cap Y| +  |N(x',y') \cap Y| & \ge  2 \delta - |X|-2\varepsilon_2' n   . \label{eqn:L(xx'y')}
\end{align}
In addition, we have
\begin{align}
	2\varepsilon_2' n +|Y| & \ge   |N(x,y) \cap Y| + |N(x,y') \cap Y| + |N(y,y') \cap Y| , \label{eqn:L(xyy')} \\
	2\varepsilon_2' n +|Y| & \ge   |N(x',y) \cap Y| + |N(x',y') \cap Y| + |N(y,y') \cap Y| \label{eqn:L(x'yy')}
\end{align}
as $|L(xyy') \cap Y |$, $|L(x'yy') \cap Y | \le \varepsilon_2' n $ respectively. 
Recall that $|X|+|Y| = n$, $|X| \le |Y|$ and $|N(y,y') \cap Y| \ge \delta -|X|$.
Together with \eqref{eqn:L(xx'y)}, \eqref{eqn:L(xx'y')}, \eqref{eqn:L(xyy')} and \eqref{eqn:L(x'yy')}, we have 
\begin{align}
	6 \delta \le & 4|X| + 2|Y|  + 8\varepsilon_2' n \le 3n + 8\varepsilon_2' n \nonumber 
\end{align}
a contradiction. 
\end{proof}

Recall that there are $(c_2-\varepsilon_2-c_1) \varepsilon_2 n^4/2$ choices of $\{x,x',y,y'\}$.
Suppose that at least $(c_2-\varepsilon_2-c_1) \varepsilon_2 n^4/8$ copies of $K_4 = \{x,x',y,y'\}$ with $|L(xx'y) \cap X| \ge \varepsilon_2' n$.
Let $u \in L(xx'y) \cap X$.
Note that $(u,y',\{x,x',y'\})$ is an $(X,Y)$-bridge.
Thus, the number of $(X,Y)$-bridges (of length 1) is at least $(c_2-\varepsilon_2-c_1) \varepsilon_2 \varepsilon_2' n^5/24$.
Therefore, we may assume without loss of generality that there are at least $(c_2-\varepsilon_2-c_1) \varepsilon_2 n^4/8$ copies of $K_4 = \{x,x',y,y'\}$ with $|L(xyy') \cap Y| \ge \varepsilon_2' n$.
Let $u \in L(xyy') \cap Y$.
Note that $(x',u,\{x,y,y'\})$ is an $(X,Y)$-bridge. 
Again, the number of $(X,Y)$-bridges is at least $(c_2-\varepsilon_2-c_1) \varepsilon_2 \varepsilon_2' n^5/24$.

(iii) There exist $c_3 n^{3}$ edges $xyy'$ of type $XYY$ such that $|L(xyy')\cap X | \ge \varepsilon_3 n$.
By~(i), we may assume that there are at least $c_3 n^3/2$ edges $xyy'$ of type $XYY$ such that $|L(xyy')\cap Y | < \varepsilon_1 n$.
Since $xyy'$ is an edge and $|L(xyy')\cap Y | < \varepsilon_1 n$, we have 
\begin{align}
	|N(x,y) \cap Y | + |N(x,y') \cap Y | + |N(y,y') \cap Y | - | Y | \le 2 | L ( xyy') \cap Y | < 2\varepsilon_1 n. \nonumber
\end{align}
Assume that $|N(x,y) \cap N(xy') \cap N(y,y') \cap X| \le  \varepsilon_3' n$ and so
\begin{align}
	|N(x,y) \cap X| + |N(x,y') \cap X| + |N(y,y') \cap X| - 2|X| \le \varepsilon_3' n. \nonumber 
\end{align}
Since $|X|+|Y| = n$ and $|X| \le n/2 \le |Y|$, (by combining the two inequalities above together) we have
\begin{align}
3 \delta & \le \deg(x,y)+\deg(x',y)+\deg(x,x') < 2 | X | + | Y | + 2 \varepsilon_1 n +\varepsilon_3' n \nonumber \\ 
& \le (3/2 + 2\varepsilon_1  + \varepsilon_3' ) n, \nonumber
\end{align}
a contradiction.
Thus, we have $|N(x,y) \cap N(x,y') \cap N(y,y') \cap X| \ge \varepsilon_3' n$.
Note that for each $u \in N(x,y) \cap N(x,y') \cap N(y,y') \cap X$, the set $\{u,x,y,y'\}$ spans a $K_4$ in $H$.
Thus, there are at least $c_3 \varepsilon_3' n^4/2 \ge c_2 n^4$ copies of $K_4$ with two vertices in each of $X$ and $Y$.
Therefore, we are done by~(ii).

(iv) There exist $c_4 n^{3}$ edges $xx'y$ of type $XXY$ such that $|L(xx'y)\cap Y | \ge \varepsilon_4 n$.
Hence, there are at least $ c_4 \varepsilon_4 n^4/2$ copies of $K_4^-$ of type $XXYY$.
Since every $K_4^-$ of type $XXYY$ contains an edge of type $XYY$, there are at $c_3 n^{3}$ edges $xyy'$ of type $XYY$ such that $|L(xyy')\cap X | \ge \varepsilon_3 n$.
Otherwise, the number of $K_4^-$ of type $XXYY$ is at most
\begin{align}
	c_3 n^3 \times n + n^3 \times \varepsilon_3 n < c_4  \varepsilon_4 n^4/2, \nonumber
\end{align}
a contradiction.
Thus, we are in case~(iii).
\end{proof}

First, we consider the case when Lemma~\ref{lma:connectioncomponent} gives exactly two partition classes as its proof will form the framework for the case when there are three partition classes.

\begin{lma} \label{lma:2parts}
Let $i_X, i_Y >0$ be integers and let $\eta_X, \eta_Y , \gamma>0$ be constants. 
Let $n$ be a sufficiently large integer and let $H$ be a $3$-graph of order $n$ with $\delta_2(H) \ge (1/2+\gamma) n$.
Suppose that $V$ is partitioned into $X$ and $Y$ with $n/4 \le |X| \le n/2 \le |Y| $.
Furthermore, suppose that $X$ and $Y$ are $(i_X, \eta_X)$-closed and $(i_Y, \eta_Y)$-closed in $H$ respectively.
Then $H$ is $(i_0, \eta)$-closed for some integer $i_0 \le 3 \max\{i_X,i_Y\}+1$ and constant $\eta >0$.
\end{lma}

\begin{proof}
Write $\delta = \delta_2(H)$.
Let $c_1,c_2,c_3,c_4,\varepsilon_1, \varepsilon_2, \varepsilon'_2, \varepsilon_3,\varepsilon'_3, \varepsilon_4, \varepsilon_5, \varepsilon'_5 >0$ be sufficiently small constants satisfying the following six inequalities:
\begin{align*}
	\varepsilon_1 & <  \min\{ \varepsilon_2, \varepsilon_3 \}, &
	c_1 +\varepsilon_2  < c_2  & < c_3 \varepsilon_3',  & \\ 
	\max\{4\varepsilon'_2, 2 \varepsilon_1 +\varepsilon'_3 \} & < 3 \gamma, & 
	2c_1  < c_3 & <  \min\{ c_4\varepsilon_4/2 - \varepsilon_3, 2^{-11} \varepsilon'_5 -\varepsilon_3\}, &\\
	\varepsilon_5  &\le  \gamma /384, &
	\varepsilon'_5 & <  1/10.
\end{align*}
Hence, they also satisfy the hypothesis of Lemma~\ref{lma:XYbridge}.
In addition, throughout this proof, $\eta_1$, $\eta_2$, \dots is assumed to be a decreasing sequence of strictly positive sufficiently small constants.
Recall that an $(X,Y)$-bridge of length~$i$ is a triple $(x,y,S)$ such that $x \in X$, $y \in Y$ and $S$ is an $(x,y)$-connector of length $i$.
By Lemma~\ref{lma:bridge}, to prove the lemma it is enough to show that there are at least $\varepsilon n^{4i+1}$ $(X,Y)$-bridges of length~$i$ for some $i, \varepsilon >0$.
We may further assume that none of conditions (i)--(iv) in Lemma~\ref{lma:XYbridge} holds, otherwise we are done.
Recall that $n/4 \le |X| \le n/2 \le |Y|$.
For every pair of vertices $x,x' \in X$, $|N(x,x') \cap Y| \ge \delta - |X| \ge \gamma n$ and so $e(XXY) \ge \binom{|X|}{2} ( \delta - |X|) \ge \gamma n^3/32$, where we recall that $e(V_1V_2V_3)$ is the number of edges of type $V_1V_2V_3$. 
Similarly, $e(XYY) \ge |X||Y| (\delta - |X|)/2 \ge \gamma n^3/32$ as $|N(x,y) \cap Y| \ge \delta - |X| \ge \gamma n$ for $x \in X$ and $y \in Y$.
In summary,
\begin{align*}
e(XXY) , e(XYY) & \ge  \gamma n^3/32.
\end{align*}
Further recall Proposition~\ref{prp:edgeextension} that $|L(e)| \ge (1/4+\gamma)n$ for all edges~$e$.
Since neither condition (i) nor (iv) in Lemma~\ref{lma:XYbridge} holds and $e(XXY) \ge \gamma n^3/32$, there are at least $ \gamma n^4/384 \ge \varepsilon_5 n^4$ copies of $K_4^-$ of type $XXXY$.
Similarly, there are at least $ \gamma n^4/384 \ge \varepsilon_5 n^4$ copies of $K_4^-$of type $XYYY$ as neither condition (i) nor (iii) in Lemma~\ref{lma:XYbridge} holds and $e(XYY) \ge \gamma n^3/32$.
Next, we split the argument into cases depending on the number of $K^-_4$ of types $XXXX$ and $YYYY$.

(a) There are $c' n^4$ copies of $K_4^-$ of type $XXXX$, where $c'$ is the constant defined in Corollary~\ref{cor:supersaturation}.
Let $m_X= 4i_X-1$ and $m_Y = 4i_Y-1$.
Recall that there are at least $\varepsilon_5 n^4$ copies of $K_4^-$ of type $XXXY$.
Pick two vertex-disjoint $K_4^-$, $T =\{x_1,x_2,x_3,x_4\}$ of type $XXXX$ and $T' = \{x'_1,x'_2,x'_3,y'\}$ of type $XXXY$.
Since $x_1$ is $(i_X, \eta_X)$-close to $x'_1$, there exist at least 
\begin{align}
\eta_X n^{m_X} - 8n^{m_X-1} \ge \eta_X n^{m_X}/2 \nonumber
\end{align}
copies of $(x_1,x'_1)$-connectors $S_1$ with $S_1 \cap (V(T) \cup V(T')) = \emptyset$.
Fix one such $S_1$.
Similarly, for $i = 2,3$ we can find an $(x_i,x'_i)$-connector $S_i$ such that $S_i \cap (V(T) \cup V(T') \cup S_1) = \emptyset$ and $S_2 \cap S_3 = \emptyset$.
Furthermore, there are at least $(\eta_X n^{m_X}/2)^2$ choices for the pair $(S_2,S_3)$.
Set 
\begin{align*}
S = S_1 \cup S_2 \cup S_3 \cup \{x_1,x_2,x_3,x'_1,x'_2,x'_3\}.
\end{align*}
Note that there is a $K_4^-$-factor in $H[S \cup y']$ as there is a $K_4^-$-factor in each of $H[T]$ and $H[x'_i \cup S_i]$ for $i = 1,2,3$.
Also, there is a $K_4^-$-factor in $H[S \cup x_4]$.
Thus, $(x_4,y',S)$ is an $(X,Y)$-bridge of length $3i_X+1$.
Moreover, there are $\varepsilon_5 c' \eta_X^3 n^{3m_X +8}/(32(3m_X +8)!)$ such $(X,Y)$-bridges.

(b) There are $c' n^4$ copies of $K_4^-$ of type $YYYY$.
We are done by an argument similar to the one used in (a).

(c) Neither (a) nor (b) holds.
By Corollary~\ref{cor:supersaturation}, we have $e(H[X]) \le 0.3 \binom{|X|}{3}$ and $e(H[Y]) \le 0.3 \binom{|Y|}{3}$.
Thus, 
\begin{align}
	e(XXY) \ge (\delta - 0.3 |X|) \binom{|X|}{2} \textrm{ and } e(XYY) \ge (\delta - 0.3 |Y|) \binom{|Y|}{2}. \nonumber
\end{align}
For $x,x' \in X$ and $y,y' \in Y$, define $a(x,x',y,y')$ to be the number of edges in $H[\{x,x',y,y'\}]$.
Note that if $a(x,x',y,y') \ge 3$, then $H[\{x,x',y,y'\}]$ contains a~$K_4^-$.
We sum $a(x,x',y,y')$ over all $x,x' \in X$ and $y,y' \in Y$, so each edge of type $XXY$ (and $XYY$)  is counted $|Y|-1$ (and $|X|-1$) times, i.e.
\begin{align}
 \sum a(x,x',y,y') & =  (|Y|-1) e(XXY) + (|X|-1) e(XYY) \nonumber \\
	& \ge  \frac12 (|X|-1)(|Y|-1) (\delta (|X|+|Y|) -0.3(|X|^2 + |Y|^2)) \nonumber \\
	& =  \frac12 (|X|-1)(|Y|-1) (\delta n -0.3(|X|^2 + |Y|^2)). \label{eqn:sum_a(x,x',y,y')}
\end{align}
If $\sum a(x,x',y,y') > (2 +4 \varepsilon'_5) \binom{|X|}{2} \binom{|Y|}{2}$, then there are at least $\varepsilon'_5\binom{|X|}{2} \binom{|Y|}{2} \ge 2^{-10}\varepsilon'_5 n^4 $ copies of 4-sets $\{x,x',y,y'\}$ such that $e(H[\{x,x',y,y'\}]) = a(x,x',y,y') \ge 3$ as $|X|,|Y| \ge n/4$.
Note that $H[\{x,x',y,y'\}]$ contains a $K_4^-$.
By an averaging argument there are at least $(2^{-11}\varepsilon'_5 -\varepsilon_3)n^3 \ge c_3 n^3$ edges $e$ of type $XYY$ with $|L(e) \cap X| \ge \varepsilon_3 n$.
This implies that condition (iii) in Lemma~\ref{lma:XYbridge} holds, a contradiction. 
Thus, we may assume that $\sum a(x,x',y,y') \le (2 +4 \varepsilon'_5) \binom{|X|}{2} \binom{|Y|}{2}$.
Recall that $n/4 \le |X| = n - |Y|$ and $\delta \ge n/2$.
Therefore, \eqref{eqn:sum_a(x,x',y,y')} becomes
\begin{align}
	(2 +4\varepsilon'_5 ) \binom{|X|}{2} \binom{|Y|}{2} & \ge   \frac12 (|X|-1)(|Y|-1) (\delta n -0.3(|X|^2 + |Y|^2)), \nonumber \\
		(1 +2 \varepsilon'_5 )  |X||Y| & \ge   \delta n - 0.3 ( |X|^2 + |Y|^2),\nonumber \\
		\varepsilon'_5   n^2& \ge   n^2/2 - 0.3 ( |X|^2 + |Y|^2) - |X||Y| \nonumber \\
	& =    n^2/10 + 0.4 ( |X| -n/2)^2 \ge n^2/10, \nonumber
\end{align}
a contradiction.
This completes the proof of Lemma~\ref{lma:2parts}.
\end{proof}

We now consider the case when $V(H)$ is partitioned into 3 classes, $X'$, $Y'$ and $Z'$ such that $|X'|,|Y'|,|Z'| \ge (1/4 + \gamma ) n$ and $X'$, $Y'$ and $Z'$ are $( \lceil 4/ \gamma \rceil +2, \eta)$-closed in $H$.
Its proof is based on the proof of Lemma~\ref{lma:2parts}.

\begin{lma} \label{lma:3parts}
Let $\gamma > 0$ and let $H$ be a $3$-graph of order $n$ with $\delta_2(H) \ge (1/2+\gamma) n$.
Suppose that $V(H)$ is partitioned into $X'$, $Y'$ and $Z'$ with $|X'|, |Y'|, |Z'| \ge n/4$ and $X'$, $Y'$ and $Z'$ are $(i_{X'}, \eta_{X'})$-closed, $(i_{Y'}, \eta_{Y'})$-closed and $(i_{Z'}, \eta_{Z'})$-closed in $H$ respectively.
Then $H$ is $(i, \eta)$-closed for some integer $i \ge 1$ and constant $\eta >0$.
\end{lma}

\begin{proof}
Write $\delta = \delta_2(H)$.
Let $m_{X'} = 4i_{X'} -1$, $m_{Y'} = 4i_{Y'} -1$ and $m_{Z'} = 4i_{Z'} -1$.
Let $c_1,c_2,c_3,c_4,\varepsilon_1, \varepsilon_2, \varepsilon'_2, \varepsilon_3,\varepsilon'_3, \varepsilon_4, >0$ be constants as defined in the proof of Lemma~\ref{lma:2parts} with an extra constant $\varepsilon_0 > 0$.
Further assume that
\begin{align*}
	\gamma \ge \max \{ 768  \varepsilon_0, 2^{12} (c_3 +\varepsilon_3) \}.
\end{align*}
Again, $\eta_1$, $\eta_2$, \dots is assumed to be a decreasing sequence of strictly positive sufficiently small constants.

A triple $(u,v,S)$ is an \emph{$i$-bridge} if it is either an $(X',Y')$-bridge, an $(X',Z')$-bridge or a $(Y',Z')$-bridge of length~$i$.
If the number of $i$-bridges is at least $\varepsilon n^{4i+1}$ for some constants $\varepsilon>0$, then we may assume without loss of generality that the number of $(X',Y')$-bridges is at least $\varepsilon n^{4i+1}/3$ .
Hence, $X' \cup Y'$ is $(i_{X'}+i_{Y'}+i)$-closed in $H$ by Lemma~\ref{lma:bridge} and so $H$ is $i_0$-closed by Lemma~\ref{lma:2parts} for some~$i_0$.
Therefore, to prove the lemma it is enough to show that there exist an integer $i_0$ and a constant $\varepsilon>0$ such that the number of $i_0$-bridges is at leasts $\varepsilon n^{4i_0+1}$.

First, suppose that there are at least $\varepsilon_0 n^4$ copies of $K_4^-$ of each of type $X'X'Y'Z'$ and~$X'Y'Y'Z'$.
Hence, we can pick two vertex-disjoint copies of $K_4^-$, $T =\{x_1,x_2,y,z\}$ of type $X'X'Y'Z'$ and $T' = \{x',y'_1,y'_2,z'\}$ of type $X'Y'Y'Z'$.
Since $x_1$ is $(i_{X'}, \eta_{X'})$-close to $x'$, there exist at least $\eta_{X'} n^{m_{X'}}/2$ copies of $(x_1,x')$-bridges $S_{X'}$ with $S_{X'} \cap (V(T) \cup V(T')) = \emptyset$.
Fix one such $S_{X'}$.
Similarly, we can find a $(y,y'_1)$-bridge $S_{Y'}$ and a $(z,z')$-bridge $S_{Z'}$ such that $S_{Y'} \cap S_{Y'} = \emptyset$ and $(S_{Y'} \cup S_{Z'}) \cap (S_{X'} \cup V(T) \cup V(T')) = \emptyset$.
Furthermore, there are at least $\eta_{Y'} n^{m_{Y'}}/2$ and $\eta_{Z'} n^{m_{Z'}}/2$ choices for $S_{Y'}$ and $S_{Z'}$ respectively.
Set $S = S_{X'} \cup S_{Y'} \cup S_{Z'} \cup \{x_1,x',y,y'_1,z,z'\}$.
Note that $(x_4,y',S)$ is an $(X',Y')$-bridge of length $i_0 = i_X+i_Y+i_Z+1$. 
Moreover, there are $\varepsilon_0^2 \eta_{X'} \eta_{Y'} \eta_{Z'} n^{m_0}/(32(m_0!))$ such $(X',Y')$-bridges, where $m_0 = 4i_0+1$.
Hence, we may assume without loss of generality that there are less than $\varepsilon_0 n^4$ copies of $K_4^-$ of each of type $X'Y'Y'Z'$ and~$X'Y'Z'Z'$.

We now mimic the proof of Lemma~\ref{lma:2parts} by setting $X=X'$ and $Y=Y' \cup Z'$. 
Note that $|X| +|Y| = n$ and $|Y| = |Y'|+|Z'| \ge n/2 \ge |X| \ge n/4$.
Observe that an $(X,Y)$-bridge of length~$i$ is an $i$-bridge.
Hence, the lemma is proved if we can show that there are many $(X,Y)$-bridges of length~$i$.
Hence, we may further assume that none of conditions (i)--(iv) in Lemma~\ref{lma:XYbridge} holds, otherwise we are done.

Since condition~(iii) does not hold, there are less than $(c_3 +\varepsilon_3)n^4$ copies of $K_4^-$ of type $XXYY$.
Therefore, there are less than $(c_3 +\varepsilon_3)n^4 < \varepsilon_0 n^4$ copies of $K_4^-$ of type $X'X'Y'Z'$.
Recall that there are less than $\varepsilon_0 n^4$ copies of $K_4^-$ of each of type $X'Y'Y'Z'$ and~$X'Y'Z'Z'$.
Thus, there are less than $3 \varepsilon_0 n^4$ copies of $K_4^-$ that contain an edge of type~$X'Y'Z'$.
Since $|L(e)| \ge (1/4 + \gamma) n$ for every edge $e$ by Proposition~\ref{prp:edgeextension}, 
\begin{align}
e(X'Y'Z') \le 24 \varepsilon_0 n^3. \nonumber
\end{align}
Without loss of generality, we may further assume that $|X'| \le |Y'| \le |Z'|$.
Let $|X'|+|Y'| = \alpha n$, so $1/2 \le \alpha \le 2/3$.
Since $(|X'|+ |Y'|)+(|X'|+|Z'|) \ge 2\alpha n$ and $|X'|+|Y'| +|Z'|= n $, we have
\begin{align}
|X'| \ge (2\alpha -1) n. \label{eqn:|X|}
\end{align}
Recall that $\gamma \ge 768 \varepsilon_0$ and $\delta \ge (1/2 + \gamma) n$.
Hence, 
\begin{align*}
e(X'Y'Y')  & =  \frac12 \left ( \sum_{x \in X', y \in Y'} (\deg(x,y) - |X'|+1) - e(X'Y'Z') \right)\\
& \ge  |X'||Y'| (\delta - |X'|+1)/2 - 12 \varepsilon_0 n^3 \\
& \ge |X'||Y'| ( (1+\gamma)n - 2|X'|)/4
\end{align*}
Similarly, we have 
\begin{align}
e(X'X'Y') & \ge  |X'||Y'| ( (1+\gamma) n - 2|Y'|)/4. \nonumber
\end{align}
For $x,x' \in X'$ and $y,y' \in Y'$, define $a(x,x',y,y')$ to be the number of edges in $H[\{x,x',y,y'\}]$ as before.
Therefore,
\begin{align}
	& \sum  a(x,x',y,y') \nonumber \\
 & =  (|Y'|-1) e(X'X'Y') + (|X'|-1) e(X'Y'Y') \nonumber \\
	& \ge 
	\frac{|X'||Y'| }4 \left[ (|Y'|-1) ( (1+\gamma)n - 2|Y'| ) +  ( |X'|-1 )  ( (1+\gamma)n - 2|X'|)      \right] \nonumber\\
	& \ge 
% 	\frac{|X'||Y'| }4 \left(  ( (1+\gamma/2 )n - 2|Y'| )|Y'| +  ( (1+\gamma/2)n - 2|X'|)|X'|      \right) \nonumber\\
%	& = 
\frac{|X'||Y'| }4 \left[ (1+\gamma/2 )n ( |X'|+|Y'| )  -2 (|X'|^2  +|Y'|^2) \right]\nonumber\\
	& = \frac{|X'||Y'| }4 \left[  4 |X'| |Y'| - ( |X'|+|Y'| )( 2|X'|+2|Y'| - (1+\gamma/2 )n ) \right] \nonumber \\
	& = \frac{|X'||Y'| }4 \left[ 4 |X'| |Y'| - \alpha ( 2\alpha - 1 -\gamma/2 )n^2 \right],
 \label{eqn:3partlast1}
\end{align}	
where we recall that $|X'|+|Y'| = \alpha n$.
Note that if $a(x,x',y,y') \ge 3$, then $H[\{x,x',y,y'\}]$ contains a~$K_4^-$.
Since there are less than $(c_3 +\varepsilon_3)n^4$ copies of $K_4^-$ of type $XXYY$, 
\begin{align*}
\sum a(x,x',y,y') \le (2 + 2^{12} (c_3 +\varepsilon_3) ) \binom{|X'|}{2} \binom{|Y'|}{2} \le (1+ \gamma) |X'|^2 |Y'|^2 /2
\end{align*}
 as $|X'|, |Y'| \ge n/4$ and $\gamma \ge  2^{12} (c_3 +\varepsilon_3)$.
Together with~\eqref{eqn:3partlast1}, we have
\begin{align}
		2(1-  \gamma ) |X'||Y'|	 & \le   \alpha ( 2\alpha - 1 -\gamma/2 )n^2 \label{eqn:3partlast}
\end{align} 
Recall~\eqref{eqn:|X|} that $|X'| \ge (2\alpha -1) n$ and $|X'|+|Y'|= \alpha n$.
Therefore, by taking $|Y' | = \alpha n - |X'|$ and $|X'| = (2\alpha -1) n$, \eqref{eqn:3partlast} becomes
\begin{align*}
		 2(1-  \gamma ) (2\alpha-1) (1-\alpha) n^2 	 & \le  \alpha ( 2\alpha - 1 -\gamma/2 )n^2,  \\
		2(1-\alpha)  &  < 2 \alpha,
\end{align*} 
 a contradiction, where  $( 2\alpha - 1 -\gamma/2 ) < (1- \gamma) (2\alpha-1)$ and $1/2 \le \alpha \le 2/3$.
The proof of Lemma~\ref{lma:3parts} is complete.
\end{proof}

Therefore, Lemma~\ref{lma:(i,eta)-close} follows immediately from Lemma~\ref{lma:connectioncomponent}, Lemma~\ref{lma:2parts} and Lemma~\ref{lma:3parts}.

\section{Closing remarks}

We would like to know the exact value of $t_2^3(n,K_4^-)$.
If Conjecture~\ref{conjecture} is true, then by Remark~\ref{rmk:lowerimproved} we know that there is no unique extremal graph for $n = 1 \pmod{3}$.
However, each of the given constructions contains $H_{n/2-1,n/2}$ as an induced subgraph.

Another natural question is to ask for the $\delta_2(H)$-threshold for the existence of $K_4^-$.
Take a random tournament on $n$ vertices, let $H$ be a 3-graph on the same vertex set such that every edge in $H$ is a directed triangle.
Note that $H$ is $K_4^-$-free and $\delta_2(H) = (1/4+o(1))n$. 
\begin{qns}
For $\varepsilon >0$, do all $3$-graphs of sufficiently large order $n$ with $\delta_2(H) \ge (1/4+\varepsilon)n$ contain a $K_4^-$?
\end{qns}
Note that a 3-graph $H$ of order $n$ with $\delta_2(H) \ge \gamma n$ contains at least $\gamma \binom{n}3$ edges. 
Thus, one of the results of Baber and Talbot~\cite{MR2769186} implies that the answer to the question above would be affirmative for $\delta_2(H) \ge (0.2871 +o(1) )n$.

\section{Acknowledgment}
The authors would like to thank the anonymous referees for the helpful comments and pointing out errors in the earlier version of the manuscript.

\end{document}